\documentclass[]{amsart}
\usepackage{graphicx,amsmath,amsthm,amsrefs}
\usepackage{color}
\definecolor{black}{gray}{0}
\definecolor{ttgray}{gray}{0.5}
\definecolor{bfred}{rgb}{0.4,0,0}
\definecolor{emgreen}{rgb}{0,0.3,0}
\definecolor{emmagenta}{rgb}{0.6,0.07,0.07}
\definecolor{sfgb}{rgb}{0,0.3,0.3}
\definecolor{mathblue}{rgb}{0,0,0.4}

\newcommand{\chegakh}[1]{\{\kern -.25em [#1]\kern -.25em \}}
\newcommand{\Dchegakh}[1]{\{\kern -.3em\{\kern -.25em [#1]\kern -.25em\}\kern -.3em\}}

\providecommand{\Im}{\loglike{lm}}

\newtheorem{ex}{Example}
\title
{An Iterative  Wiener--Hopf method for triangular matrix functions with exponential factors}
\author{Anastasia V. Kisil}    
\address{DAMTP, University of Cambridge,
  Wilberforce Road, Cambridge, CB3 0WA, UK }
\email{a.kisil@maths.cam.ac.uk}

        \newtheorem{thm}{Theorem}[section]

    \PassOptionsToPackage{british}{babel}

    \theoremstyle{definition}

    \theoremstyle{remark}
    \newtheorem{rem}[thm]{Remark}

\begin{document}

 \begin{abstract}
   This paper introduces a new method for constructing approximate
   solutions to a class of Wiener--Hopf equations. This is
   particularly useful since exact solutions of this class of
   Wiener--Hopf equations, at the moment, cannot be obtained. The
   proposed method could be considered as a generalisation of the pole
   removal technique. The error in the approximation can be explicitly
   estimated, and by a sufficient number of iterations could be made
   arbitrary small. Typically only a few iterations are required for
   practical purposes.  The theory is illustrated by numerical
   examples that demonstrate the advantages of the proposed
   procedure. This method was motivated and successfully applied to
   problems in acoustics.
\end{abstract} 

\keywords
{Wiener--Hopf Equations,  Riemann-Hilbert Problem, Iterative Methods}  

\maketitle
\
\section{Introduction} 
\label{sec:Introduction}

Many boundary value problems in mathematical physics can be approached
by the Wiener--Hopf method. Originally the Wiener--Hopf technique was
developed for linear PDEs with semi-infinite boundary conditions like
the Sommerfeld half-plane problem. This was later extended to boundary
conditions on multiple semi-infinite lines~\cite{history}. Although,
the reduction to the Wiener--Hopf equation is still straightforward,
finding a
solution of the resulting equation became
challenging~\cite{Constructive_review, Ant_crack, Mer,
  Daniele,Fokas_uni, Crowdy-Elena}. This is due to the fact that the
Wiener--Hopf factorisation of a matrix (rather than a scalar) function is now needed. Hence the ability to solve such equations is
crucial to extending the classical use of Wiener--Hopf techniques to
more realistic and complicated settings. Also, such types of matrix
Wiener--Hopf equations are associated with convolution-type operators on
a finite interval~\cite{Feldman2000, Feldman2004, Pure_rew_2013} and arise in a number of applications
\cite{Nigel_cyl,Abr_exp}.  The aim of this paper is to develop an
algorithmic iterative method of solution for some equations of this
type.

More precisely, we construct an approximate solution of a Wiener--Hopf
equation with triangular matrix functions containing exponential
factors. The aim is to find functions \(\Phi_-^{(0)}(\alpha)\),
\(\Phi_-^{(L)} (\alpha)\),
\(\Psi_+^{ (0)}(\alpha)\)
and \(\Psi_+^{ (L)}(\alpha)\)
analytic in respective half-planes, satisfying the following
relationship
\begin{equation}
  \label{eq:main}
\begin{pmatrix}
 \Phi_-^{(0)}(\alpha) \\
 \Phi_-^{(L)} (\alpha)
 \end{pmatrix}=\begin{pmatrix}
  A(\alpha)&  B(\alpha)e^{i \alpha L} \\
 C(\alpha)e^{-i \alpha L} & 0
 \end{pmatrix}\begin{pmatrix}
 \Psi_+^{ (0)}(\alpha) \\
 \Psi_+^{ (L)}(\alpha) 
 \end{pmatrix}+ \begin{pmatrix}
 f_1(\alpha)\\
f_2 (\alpha)
 \end{pmatrix},
\end{equation}
on the strip \(a\le \Im(\alpha)\le b\). The remaining functions \(A(\alpha)\), 
\(B(\alpha)\) and \(C(\alpha)\) are known and \(L\) is a positive
constant.  The conditions on the matrix functions are
specified in Section \ref{sec:Preliminaries}. The existence of a such
factorisation under certain assumptions was addressed in
\cite[p.~150]{Spit_book} and \cite{Existance_osi}.

Wiener--Hopf equations of the type~\eqref{eq:main} have been the topic
of previous research, for example, in the case of meromorphic matrix
entries \cite{Antipov, Aktosun1992} and in the framework of almost periodic
functions \cite{karlovich-spit-alm-per}. However, presently no complete solution
of \eqref{eq:main} is known. Furthermore, the
general question of constructive Wiener--Hopf factorisation is widely
open \cite{history, Constructive_review}.

In the literature, concerning applications of the Wiener--Hopf technique,
one of most widely used methods is the so-called ``pole removal'' or
``singularities matching'' \cite{bookWH}*{\S~4.4,~5.3}
\cite{daniele2014wiener}*{\S~4.4.2}. It has a severe limitation that
certain functions have to be rational or meromorphic. One way to
extend the use of this method is by employing a rational approximation
\cite{Pade, My1}, which was successfully used in
\cite{Abraha_all_pade, Ab_ex, Nigel_cyl, MyD-K}. However, even with
this extension the class of functions, which can be solved, is rather
limited. In this paper we propose a different extension to the pole
removal technique for functions that have arbitrary
singularities. This work was motivated by certain problems in
acoustics which are discussed below.

We note, that there is an additional difficulty in finding a
factorisation of the matrix in \eqref{eq:main} because of the presence
of analytic functions \(e^{i \alpha L}\)
and \(e^{-i \alpha L}\)
which have exponential growth in one of the half-planes.  The first
step of the procedure proposed here is a partial factorisation, that
has at most polynomial growth in the respective half-planes. The
classes for which this is a complete factorisation are also
discussed. The next step is the Wiener--Hopf additive splitting of
the remainder term that hinders the application of Liouville's
theorem. The additive Wiener--Hopf splitting is routine unlike the
multiplicative one, which is also utilised in other novel methods
\cite{Mishuris-Rogosin, assym_2}.  After the application of the analytic
continuation there are still some unknowns in the formula, those are
approximated by an iterative procedure. The presence of the
exponential terms speeds up convergence. At each step a scalar
Wiener--Hopf equation is solved.
We will compare the proposed method to the pole removal technique. It
is shown that the iterations converge quickly to the exact solution.

The procedure could be summarised as follows (see
Section~\ref{sec:Iterative} for details):
\begin{enumerate}
\item A partial factorisation
with exponential factors in the desired half-planes.
\item Additive splitting of some terms.
\item Application of Liouville's theorem.
\item Iterative procedure to determine the remaining unknowns.
\end{enumerate}
This procedure bypasses the need to construct a multiplicative matrix
factorisation. So in particular partial indices (which are know to be
linked to stability~\cite{Spit_book,MyD-K}) are not obtained. Instead
the growth at infinity of certain terms play a role. In this paper we
will treat the base case with no growth at infinity and some other
cases will be treated in~\cite{my+Lorna}.

The structure of the paper is as follows. In
Section~\ref{sec:Preliminaries} the required classes of functions are
introduced and their essential properties are listed.  We also provide
some motivation behind those Wiener--Hopf systems.  In 
Section~\ref{sec:Iterative} the proposed iterative procedure is
described in detail and its convergence is examined in
Section~\ref{sec:Convergence}. Section~\ref{sec:ex} presents numerical
results of two examples (graphically illustrated) to compare the iterative procedure
to the exact solution in a variety of cases. Lastly we describe
possible future work.


\section{Preliminaries}
\label{sec:Preliminaries}

In order to formulate the problem \eqref{eq:main} we have to specify
the suitable class of functions for the all terms in the
equation. Those classes of functions are convenient for formulating
scalar Wiener--Hopf splittings which will form the foundation of the
proposed method. We will also need some properties of this class
of function which will ensure we stay the the desired class of
functions after each iterations. We will also briefly describe the
motivation behind this method.  Firstly, we will recall some definitions.

\subsection{Classes of functions \(\Dchegakh{a,b}\)}
\label{sec:1}

The H\"{o}lder continuous
functions on the compactified real line are defined as functions \(F(x)\) for
which there exist constants \(C\) and \(\lambda\) such that for all
real \(x_1\) and \(x_2\) we have
\begin{displaymath}
  |F(x_1 )- F(x_2 )| \leq C|x_1 - x_2 |^\lambda , \quad \text{ for all }
  x_1,x_2\in \mathbb{R}
\end{displaymath}
and for all real \(x_1\) and \(x_2\) with modulus greater than one the
following holds
\begin{displaymath}
  |F(x_1 )- F(x_2 )| \leq C\left|\frac{1}{x_1} - \frac{1}{x_2}  \right|^\lambda ,
\end{displaymath}
which is a H\"older condition around infinity. In
the following, the ``real line'' will always mean the ``compactified real line''.

A H\"{o}lder continuous
function produces  well-defined boundary values
\(G_\pm(t)=\lim_{y\searrow 0} G(t\pm iy )\) of its Cauchy type
integral \(G(t+iy)\) at every point \(t\) of the
compactified real line~\cite{Musk}. The class
\(L_2(\mathbb{R)}\) is also very useful since the Fourier transform is an 
isometry of \(L_2(\mathbb{R)}\) due to Plancherel's theorem.
Thus, the intersection~\cite{Ga-Che}*{\S~1.2}:
\[ \Dchegakh{0}=L_2(\mathbb{R}) \cap \text{ H\"{o}lder},\] 
turns out to be convenient for the Wiener--Hopf
problems.  The pre-image of \(\Dchegakh{0}\) under the Fourier
transform is denoted \(\chegakh{0}\).

 Given a function \(f\in\chegakh{0}\) on the real line we can define the splitting
\begin{equation}
\label{eq:plus}
 f_+(t)= \left\{
 \begin{array}{rl}
 f(t) & \text{if } t>0,\\
  0 & \text{if } t<0,\\ 
 \end{array} \right. 
\quad
 f_-(t)= \left\{
 \begin{array}{rl}
 0 & \text{if } t>0,\\
 -f(t) & \text{if } t<0.\\ 
 \end{array} \right.
\end{equation}
Using this splitting, we define the class \(\chegakh{0,\infty}\) to contain functions \(f_+(t)\) and
the class \(\chegakh{-\infty,0}\) to contain  functions \(f_-(t)\) for all \(f\in\chegakh{0}\).

In the rest of the paper we will need to refer to functions that are
analytic on strips or (shifted) half-planes. 
Following~\cite{Ga-Che}*{\S~13}, we define
\(f\in \chegakh{a}\) if \(e^{-ax}f\in \chegakh{0}\), that is a shift in the Fourier
space. Finally, \(f=f_++f_-\in \chegakh{a,b}\) if \(f_+\in \chegakh{a}\) and
\(f_-\in \chegakh{b}\). From the definition of \(f_+\) and \(f_-\) it is
clear that also \(f_+\in \chegakh{a, \infty}\) and \(f_-\in \chegakh{-\infty,b}\).

The Fourier transform of functions in the class \(\chegakh{a,b}\) is denoted
\(\Dchegakh{a,b}\) (including the case \(a=-\infty\) and/or \(b=\infty\)). The celebrated Paley--Wiener theorem states that
the following 
\begin{eqnarray}
\label{eqn:ana1}
F_+(z)\in\Dchegakh{a,\infty} &\Longrightarrow &  F_+(z) \text{ analytic in Im}z>a ,\\
\label{eqn:ana2}
  F_-(z)\in\Dchegakh{-\infty,b}  &\Longrightarrow& F_-(z) \text{ analytic in Im}z<b,\\
\label{eqn:ana3}
F(z)\in\Dchegakh{a,b} &\Longrightarrow& F(z) \text{ analytic in } a<\mathrm{Im}\,z<b.
\end{eqnarray}
For the remainder of the paper we will
assume that \(a<0\) and \(b>0\), i.e. the strip encloses the real
axis.

Now we will recall the additive and multiplicative scalar splitting of
functions belonging to \(\Dchegakh{a,b}\). A convention will be used to distinguish the additive and the
multiplicative splitting by using the superscript and subscript
notations: e.g \(F^\pm\) for additive and \(K_\pm\) for multiplicative
ones. 
\begin{thm}[Additive splitting~\cite{Ga-Che}]
  \label{prob:jump-probl}
  On the real line a function \(F(t)\in \chegakh{0}\) is
  given. There exist two functions \(F^{\pm}(z)\) analytic
  in the upper and lower half-planes with boundary functions on the
  real line belonging to the classes \(\Dchegakh{0,\infty}\) and
   \(\Dchegakh{-\infty, 0}\), satisfying:
\begin{equation*}
  F(t)=F^{+}(t)+F^-(t), 
\end{equation*}
on the real line.  
\end{thm}

Next, the multiplicative splitting or factorisation problem is
examined. The \emph{index} of a continuous non-zero function \(K(t)\)
on the real line is the winding number of the curve
\((\text{Re } K(t), \mathrm{Im}\, K(t)) \),
\(t \in \mathbb{R}\).

\begin{thm}[Multiplicative splitting~\cite{Ga-Che}]
  \label{prob-R-H-factor}
  Let a non-zero function \(K(t)\), such that  \(K(t)-1\in
  \chegakh{0}\) and \(\mathrm{ind}\, K(t)=0\), be   given. There exist two
  functions \(K_{\pm}(z)\) analytic in the upper and lower half-planes respectively
  with boundary functions \(K_{\pm}(t)-1\) on the real line belonging to the
  classes \(\Dchegakh{0,\infty}\) and \(\Dchegakh{-\infty, 0}\),
  satisfying:
  \begin{equation*}
    K(t)=K_{+}(t)K_-(t),
  \end{equation*}
  on the real line.
\end{thm}

\begin{rem}
  Note, that if the original function in Theorem \ref{prob:jump-probl}
  or Theorem \ref{prob-R-H-factor} is from the class \(\chegakh{a}\)
  then the components belong to the classes \(\Dchegakh{a,\infty}\)
  and \(\Dchegakh{-\infty, a}\).
\end{rem}

With those definitions we can specify the classes of functions in
equation \eqref{eq:main}.  The given functions \(A(\alpha)-1\),
\(C(\alpha)-1\),
\(B(\alpha)-1\)
are required to be in \(\Dchegakh{a,b}\)
and \(A(\alpha)\),
\(C(\alpha)\),
\(B(\alpha)\)  to have zero index.  We also need that \(C(\alpha)\)
and \(B(\alpha)\)
have no zeroes on the strip---this corresponds to the determinant
of the original Wiener--Hopf matrix \eqref{eq:main} being non-zero.  We will also
require that \(A(\alpha)\) has no zeroes on the strip, this
corresponds to the sub-problem of \eqref{eq:main} (when \(L\) is very large) being
non-degenerate.   The ``forcing terms''
\(f_1(\alpha)\)
and \(f_2(\alpha)\)
should be in \(\Dchegakh{a,b}\).
We look for \(\Phi_-^{(0)}(\alpha) \)
and \(\Phi_-^{(L)}(\alpha) \)
in \(\Dchegakh{-\infty,b}\)
and \(\Psi_+^{ (0)}(\alpha)\)
and \(\Psi_+^{ (L)}(\alpha)\) in \(\Dchegakh{a,\infty}\).

\subsection{Properties of functions in \(\Dchegakh{a,b}\)}
\label{sec:Properties}
In this subsection we briefly describe some of the properties of
\(\Dchegakh{a,b}\) which will ensure that the solution we seek is in
the desired class. 
We will need the fact that if \(G\in\Dchegakh{a,b}\), \(F\in\Dchegakh{\alpha,\beta}\) (\(a<b\),
\(\alpha<\beta\))  and \(H(z)=F(z) G(z)\) then
\begin{equation*}
H(z) \in\Dchegakh{\text{max}(a, \alpha), \text{min}(b, \beta)},
\end{equation*}
i.e \(H(z)\)
is analytic in the respective strip. This is a simple but still
fundamental property for the rest of the derivation. Note, that a
similar statement does not hold for the class of \(L_2(\mathbb{R})\).

We will also use that if \(G-1\in\Dchegakh{a,b}\), \(\mathrm{ind}\,
G(t)=0\) and non-zero on the strip \(a\le \Im(\alpha)\le b\) then \(G^{-1}-1\in\Dchegakh{a,b}\).

Note that in the proposed method we only
solve scalar Wiener-Hopf equations of the similar form to the ones detailed
in~\cite{Ga-Che}*{\S~3.2}.  That is for a scalar equations on the real line
\begin{equation}
\label{eq:W-H3}
 K(t)\Phi_-(t)= \Phi_+(t)+ F(t).
 \end{equation}
we need that \(K\) is non-zero function, such that  \(K(t)-1\in
  \chegakh{0}\) and \(\mathrm{ind}\, K(t)=0\). Also require that \(F(t)\in
  \chegakh{0}\) then we will can find \(\Phi_+\) and \(\Phi_-\) in
  \((\Dchegakh{0,\infty}\) and \(\Dchegakh{-\infty, 0}\) on the real axis. We will need
  the extension of those result for the case when \eqref{eq:W-H3} holds
  on a strip \(a\le \Im(\alpha)\le b\). This has been obtained in \cite{MyStrip}.

\subsection{Motivation}
\label{sec:Motivation}

The initial motivation for this work came from the following acoustics
problem: investigate the effect of a finite poroelastic trailing edge on
noise production \cite{owl,Lorna}. The situation is modelled with a
plane-wave scattering by a rigid half-line plate \((-\infty,0)\)
and an poroelastic edge \((0,L)\)
with the poroelastic to rigid plate transition at \(x=0\).
The matrix Wiener--Hopf problem is obtained from consideration of the
Fourier transform with respect to the ends of the poroelastic plate at
\(x=0\)
and \(x=L\).
Define the half-range and full-range Fourier transform with respect to
\(x=0\):
\begin{align}
\Phi(\alpha, y)=& \int_{-\infty}^{0}\phi( \xi,y) e^{i \alpha \xi} d \xi
+\int^{\infty}_{0}\phi (\xi,y)  e^{i \alpha \xi} d \xi,\\
 =& \Phi^{(0)}_-( \alpha,y)+\Phi^{(0)}_+(\alpha,y).
\label{eqn:halfF}
\end{align}
In the case when there is only one transition point \(x=0\), the unknown
functions would be \(\Phi^{(0)}_-( \alpha,y)\),
\(\Phi^{(0)}_+(\alpha,y)\)
(or their derivatives) and we would only need to solve a scalar
Wiener--Hopf equation~\cite{bookWH}. Since there is a change of
boundary conditions at \(x=L\) as well,
we will also define the Fourier transforms with respect to the point
\(x=L\):
\begin{align}
\Phi^{(L)}( \alpha,y)=& \int_{-\infty}^{L}\phi ( \xi,y) e^{i \alpha (\xi-L)} d \xi
+\int^{\infty}_{L}\phi ( \xi, y)  e^{i \alpha (\xi-L)} d \xi,\\
 =& \Phi^{(L)}_-( \alpha,y)+\Phi_{+}^{(L)}(\alpha, y).
\end{align}
The relation between the transforms is:
\[ \Phi^{(L)} (\alpha,y)=\Phi ( \alpha,y) e^{i \alpha L}.\]
The next step is to write down the relationship between different
half-range transforms by using the boundary
conditions~\cite{daniele2014wiener}. These relations can be combined
to form a matrix Wiener--Hopf equation. The resulting matrix, which
motivated the present method, is (see \cite{my+Lorna} for a detailed
discussion):
\[\begin{pmatrix}
 \Phi_-^{(0)}(\alpha) \\
 \Phi_-^{(L)} (\alpha)
 \end{pmatrix}=-\begin{pmatrix}
  \frac{1-\gamma(\alpha)P}{\gamma(\alpha)}  &  Pe^{i \alpha L} \\
 \frac{1}{\gamma(\alpha)}e^{-i \alpha L} & 0
 \end{pmatrix}\begin{pmatrix}
 \Phi_+^{ '(0)}(\alpha) \\
 \Phi_+^{ '(L)}(\alpha) 
 \end{pmatrix}+ \begin{pmatrix}
 f_1(\alpha)\\
f_2 (\alpha)
\end{pmatrix},\]
where \(\gamma(\alpha)=\sqrt{\alpha^2-k^2_0}\),
\(k_0\)
is the acoustic wave number and, in the simplest case, \(P\)
is a constant. The exponentials are due to the boundary conditions
since the half-range Fourier transform with respect to different points
are needed. 

We note that on the left hand side the unknown functions are minus
half-transforms \(\Phi_-^{(0)}(\alpha)\)
and \(\Phi_-^{(L)} (\alpha)\),
and on the right are the derivatives of the plus half-transforms which
are in this paper denoted by \(\Psi_+^{(0)}(\alpha)\)
and \(\Psi_+^{ (L)}(\alpha)\). This problem is not treated in this
paper further since the purpose of this paper is to present the
iterative procedure in the simplest and general case in order to make
it easier to apply to many situations. For a detailed discussion of
this particular problem see \cite{my+Lorna}. The the later paper we
also discuss other examples, the diffraction with finite rigid plate,
leading to the following Wiener--Hopf equation.
\[\begin{pmatrix}
\Phi_-^{'(0)}(\alpha) \\
\Phi_-^{(L)} (\alpha)
\end{pmatrix}=-\begin{pmatrix}
  \gamma(\alpha) &  e^{i \alpha L} \\
-e^{-i \alpha L} & 0
\end{pmatrix}\begin{pmatrix}
\Phi_+^{ (0)}(\alpha) \\
\Phi_+^{ '(L)}(\alpha) 
\end{pmatrix}- \begin{pmatrix}
g_1(\alpha)\\
0
\end{pmatrix},\]
The advantage in considering the following example is that although
the exact matrix factorisation cannot be constructed, the results of
the new iterative procedure can be compared to the exact solution
obtained by other methods (Mathieu functions). 


\section{Iterative Wiener--Hopf Factorisation}
\label{sec:Iterative}

The most characteristic aspect of the Wiener--Hopf method is the
application of Liouville's theorem in order to obtain two separate
equations from one equation.  In order for Liouville's
theorem to be used two conditions have to be satisfied: the
analyticity and (at most) polynomial growth at infinity.  These two
conditions will be treated here in turn. First, a partial factorisation is
considered that has the exponential functions in the right place and
some of the required analyticity, that is:
\begin{align}
 \lefteqn{\begin{pmatrix}
\frac{-e^{-i \alpha L} }{B_-(\alpha)} &\frac{A(\alpha) }{C(\alpha)B_-(\alpha)}\\
 \frac{1 }{B_-(\alpha)}  & 0
 \end{pmatrix}\begin{pmatrix}
  \Phi_-^{(0)}(\alpha)\\
\Phi_-^{(L)} (\alpha)
 \end{pmatrix}}\nonumber \\ 
&\hspace*{2cm}=\begin{pmatrix}
  0  &   - B_+(\alpha) \\
\frac{ A(\alpha)}{B_-(\alpha) } &   B_+(\alpha)e^{i \alpha L}
 \end{pmatrix}\begin{pmatrix}
  \Psi_+^{ (0)}(\alpha) \\
 \Psi_+^{ (L)}(\alpha) 
 \end{pmatrix}+ \begin{pmatrix}
 f_3\\
f_4  \end{pmatrix},\label{eq:slit11} 
\end{align}
on a strip \(a\le \Im(\alpha)\le b\), where
\begin{displaymath}
  f_3=\frac{-e^{-i \alpha L} }{B_-(\alpha)}f_1+\frac{A(\alpha)
}{C(\alpha)B_-(\alpha)}f_2 \quad \text{and} \quad f_4=\frac{f_1
}{B_-(\alpha)}.
\end{displaymath}

 Note, that \(f_3\) and \(f_4\) are still in \(\Dchegakh{a,b}\).
There are two cases which would allow to solve the above equation exactly. The first case is
when   \(\frac{A(\alpha) }{C(\alpha)B_-(\alpha)}\) is in \(\Dchegakh{-\infty,b}\) and
\(\frac{A(\alpha)}{B_-(\alpha) } \) is in \(\Dchegakh{a,\infty}\). Then the matrix
factorisation has been already achieved in \ref{eq:slit11}. 
 In particular, this is true for matrices that have
the form:
\[
\begin{pmatrix}
 k B_-(\alpha) C_+(\alpha)&  B(\alpha)e^{i \alpha L} \\
 C(\alpha)e^{-i \alpha L} & 0
\end{pmatrix},\]
where \(k\)
is a constant. The second important case is when
\(\left(\frac{A(\alpha) }{C(\alpha)B_-(\alpha)}\right)^+\)
and \(\left(\frac{A(\alpha)}{B_-(\alpha) }\right)^- \)
are rational functions. Then the pole removal method 
\cite{bookWH}*{\S~4.4,~5.3} \cite{daniele2014wiener}*{\S~4.4.2} could
be employed to obtain the factorisation.

In the generic case we start with the
partial factorisation \eqref{eq:slit11} and then use the additive Wiener--Hopf
splitting. We present the detailed description now. The equation \eqref{eq:slit11} can be rearranged as
\begin{align}
 \lefteqn{\begin{pmatrix}
\frac{-e^{-i \alpha L} }{B_-(\alpha)} &\left( \frac{A(\alpha) }{C(\alpha)B_-(\alpha)}\right)^-\\
 \frac{1 }{B_-(\alpha)}  & 0
 \end{pmatrix}
\begin{pmatrix}
 \Phi_-^{(0)}(\alpha) \\
\Phi_-^{(L)}(\alpha)
 \end{pmatrix}+\begin{pmatrix}
 \left( \frac{A(\alpha) }{C(\alpha)B_-(\alpha)}\right)^+\Phi_-^{(L)}(\alpha) \\
 -\left(\frac{ A(\alpha)}{B_-(\alpha) }\right)^-  \Psi_+^{ (0)}(\alpha)
 \end{pmatrix}}\label{eq:slit1}\\ 
&\hspace*{2cm}=
\begin{pmatrix}
  0  &   - B_+(\alpha) \\
\left(\frac{ A(\alpha)}{B_-(\alpha) }\right)^+ &   B_+(\alpha)e^{i \alpha L}
 \end{pmatrix}
\begin{pmatrix}
 \Psi_+^{ (0)}(\alpha) \\
\Psi_+^{ (L)} (\alpha)
 \end{pmatrix}+
\begin{pmatrix}
 f_3\\
f_4 
\end{pmatrix}\label{eq:slit2}
,
\end{align}
on a strip \(a\le \Im(\alpha)\le b\).
Recall that the additive splitting is denoted \(F^\pm\) and the
multiplicative splitting by \(K_\pm\).
As the next step we make the additive splittings
of the second term of \eqref{eq:slit1}, which are possible since it is
in the class \(\Dchegakh{a,b}\)---in the same way as for the second term
of \eqref{eq:slit2}. Then, the  Liouville's theorem could be applied
because the exponential functions are in the correct place and all the
functions have desired the analyticity. Thus, we can apply the
Wiener--Hopf procedure as usual and the four equations (defined for
all \(\alpha\) in the complex plane) then become
\begin{align*}
\frac{-e^{-i \alpha L} }{B_-} \Phi_-^{(0)}+\left( \frac{A }{CB_-}\right)^-\Phi_-^{(L)}+\left(\left( \frac{A }{CB_-}\right)^+\Phi_-^{(L)}\right)^-- f_3^-=& 0,\\
- B_+\Psi_+^{ (L)}-\left(\left( \frac{A }{CB_-}\right)^+\Phi_-^{(L)}\right)^++ f_3^+=&0,\\
 \frac{1 }{B_-} \Phi_-^{(0)}- 
\left( \left(\frac{ A}{B_- }\right)^-  \Psi_+^{ (0)}\right)^--f_4^-=& 0,\\
\left(\frac{ A}{B_- }\right)^+ \Psi_+^{ (0)}+B_+ e^{i \alpha L}\Psi_+^{ (L)} +\left( \left(\frac{ A}{B_- }\right)^-  \Psi_+^{ (0)}\right)^-+f_4^+=&0.
\end{align*} 
Note that  Liouville's theorem is applied before any approximations are made.
The four equations can be rearranged:
\begin{align}
\left( \frac{A }{CB_-}\right)^-\Phi_-^{(L)}=&  f_3^--\left(\left( \frac{A }{CB_-}\right)^+\Phi_-^{(L)}\right)^- +\frac{e^{-i\alpha L}}{B_-} \Phi_-^{(0)}, \label{eq:eq3} \\ 
 B_+ \Psi_+^{ (L)}=&f_3^+-\left(\left( \frac{A }{CB_-}\right)^+\Phi_-^{(L)}\right)^+ \label{eq:eq1}, \\
\left(\frac{ A}{B_- }\right)^+\Psi_+^{ (0)}=& \left( \left(\frac{ A}{B_- }\right)^-  \Psi_+^{ (0)}\right)^++f_4^++B_+e^{i \alpha L}\Psi_+^{ (L)}, \label{eq:eq4}\\
\frac{\Phi_-^{(0)}}{B_-}=&\left( \left(\frac{ A}{B_- }\right)^-  \Psi_+^{
      (0)}\right)^-+f_4^-.
\label{eq:eq2}
 \end{align}
When the equations are written in this form  it is clear that
if \(\Phi_-^{(L)}\) is known then it could be used to calculate  
\(\Psi_+^{(L)}\) and this, in turn, produces \(\Psi_+^{ (0)}\) followed
by the calculation of
\(\Phi_-^{(0)}\) and then it loops round.  To avoid cumbersome
notations,  we will label coefficients in
(\ref{eq:eq3}--\ref{eq:eq2})  by \(K^{\pm}_i\) and obtain the following system:
\begin{align}
K_1^-\Phi_-^{(L)}=&  -\left(K_1^+\Phi_-^{(L)}\right)^-+f_3^- +K_4^+e^{-i
  \alpha L} \Phi_-^{(0)}, \label{eq:eq31} \\ 
K_2^+ \Psi_+^{ (L)}=& \left( K_1^+\Phi_-^{(L)}\right)^++f_3^+ \label{eq:eq11}, \\
K_3^+ \Psi_+^{ (0)}=&- \left( K_3^-  \Psi_+^{ (0)}\right)^++f_4^++K_2^+e^{i \alpha L}\Psi_+^{ (L)}, \label{eq:eq41}\\
K_4^-\Phi_-^{(0)}=&\left( K_3^-  \Psi_+^{ (0)}\right)^-+f_4^-.\label{eq:eq21}
 \end{align}
Using equations~\eqref{eq:eq11} and ~\eqref{eq:eq21}, we eliminate
\(\Psi_+^{ (L)}\) and \(\Phi_-^{(0)}\) from \eqref{eq:eq31} and
\eqref{eq:eq41} respectively:
\begin{align}
 K_1^-\Phi_-^{(L)}=&  f_3^--\left(K_1^+\Phi_-^{(L)}\right)^- +e^{-i
  \alpha L}\left(\left( K_3^-  \Psi_+^{ (0)}\right)^-+f_4^-\right) , \label{eq:eq32} \\
K_3^+\Psi_+^{ (0)}=& f_4^+-\left( K_3^-  \Psi_+^{ (0)}\right)^++e^{i \alpha L}\left( \left( K_1^+\Phi_-^{(L)}\right)^++f_3^+\right). \label{eq:eq42}
 \end{align}
Functions \(\Psi_+^{ (L)}\) and \(\Phi_-^{(0)}\) can be found
from~\eqref{eq:eq11} and ~\eqref{eq:eq21} once \eqref{eq:eq32} and
\eqref{eq:eq42} will be solved, where \(K_1=\frac{A }{CB_-}\) and
\(K_3=\frac{A }{B_-}\) in the original notation.

So far the equations are exact, but in order to make progress an
approximation will be used now.  In order to solve approximately we
will describe an iterative procedure, where the \(n\)-th
iteration is denoted by \(\Phi_-^{(L)n}\)
and \(\Psi_+^{ (0)n}\).
If \(\Phi_-^{(L)n}\)
is known it could be substituted into \eqref{eq:eq42} to calculate
\(\Psi_+^{ (0)n+1}\)
and then the function \(\Psi_+^{ (0)n+1}\)
can be used in \eqref{eq:eq32} to find \(\Phi_-^{(L)n+1}\)
and so on. 

Hence, for this iterative procedure it is enough to choose
an initial value of \(\Phi_-^{(L)0}\).
Since equation \eqref{eq:eq32} can be considered on a horizontal line
\(\Im(\alpha)=a<0\),
on that line the term with \(e^{-i\alpha L}\)
will be small (especially for \(L\)
large). This justifies neglecting the term with \(e^{-i\alpha L}\)
as a first approximation. Hence to find \(\Phi_-^{(L)0}\)
the aim is to solve
 \begin{equation}
\label{eq:sca1}
K_1^-\Phi_-^{(L)0}=  f_3^--\left(K_1^+\Phi_-^{(L)0}\right)^-.
 \end{equation}
 The above equation can be rearranged as a scalar Wiener--Hopf
 equation in the following manner
\begin{equation}
  \label{eq:first-approx}
  \left(K_1\Phi_-^{(L)0}\right)^-=  f_3^-.
 \end{equation}
Introduce an unknown function \(D^+\) defined by
\begin{equation}
  \label{eq:d-plus-def}
  D^+=\left(K_1\Phi_-^{(L)0}\right)^+.
\end{equation}
Then, combining \eqref{eq:first-approx} and \eqref{eq:d-plus-def} we
obtain a Wiener--Hopf equation:
\begin{equation}
\label{eq:W-H1}
 K_1\Phi_-^{(L)0}= D^++ f_3^-.
 \end{equation}
This equation has the form discussed in Section~\ref{sec:Properties}
so we will have \(\Phi_-^{(L)0}\in~\Dchegakh{-\infty, a}\) as desired.
The solution of this equation is
\begin{equation*}
\Phi_-^{(L)0}= \frac{1}{(K_{1})_-}\left(\frac{f_3^-}{(K_1)_+}\right)^-.
\end{equation*}
This can be taken as the initial approximation of the solution, which is used
in our iterative procedure. 

To compute the next step we will need to solve~\eqref{eq:eq42} for
\(\Psi_+^{ (0)1}\)
on a horizontal line \(\Im(\alpha)=b>0\).
Note that \(\Phi_-^{(L)1}\)
is defined on this line and can be evaluated numerically. Exactly in
the same manner as \eqref{eq:sca1}, the solution of
\begin{equation*}
K_3^+\Psi_+^{ (0)1}= \left(-\left( K_3^-  \Psi_+^{ (0)1}\right)^++f_4^+\right)+e^{i \alpha L}\left( \left( K_1^+\Phi_-^{(L)0}\right)^++f_3^+\right),
\end{equation*}
will lead to the scalar Wiener--Hopf equation now on the line
\(\Im(\alpha)=b\). It can be found that
\begin{equation}
  \label{eq:iteration-psi}
\Psi_+^{ (0)1}= \frac{1}{(K_{3})_+}\left(\frac{f_4^++e^{i \alpha L}\left( \left( K_1^+\Phi_-^{(L)0}\right)^++f_3^+\right)}{(K_3)_-}\right)^+.
\end{equation}
It is easy to see that this formula holds for all iterations with a
trivial change of \(\Psi_+^{ (0)1}\) to \(\Psi_+^{ (0)n}\) and
\(\Phi_-^{(L)0}\) to \(\Phi_-^{(L)n-1}\). Similarly, the general
recurrence formula for \(\Phi_-^{(L)n}\) is 
\begin{equation}
\label{eq:iteration-phi}
\Phi_-^{(L)n}= \frac{1}{(K_{1})_-}\left(\frac{f_3^- +e^{-i
  \alpha L}\left(\left( K_3^-  \Psi_+^{ (0)n}\right)^-+f_4^-\right) }{(K_1)_+}\right)^-.
\end{equation}

The convergence of this procedure is examined in the next
section. Numerical examples of this procedure are given in Section
\ref{sec:ex} and are compared to exact solutions (which are known for
these special cases).

\section{Convergence of the method}
\label{sec:Convergence}
The convergence of  iterations relies on consideration of
 equations~\eqref{eq:eq32} and \eqref{eq:eq42} on
different lines \(\Im(\alpha)=a<0\) and \(\Im(\alpha)=b>0\) within the
strip of analyticity \(a\le\Im(a)\le b\), in a similar way as
it was done in \cite{MyStrip}. We will employ the following notation
\begin{equation*}
D_n^+=\left(K_1^+(\Phi_-^{(L)n}-\Phi_-^{(L)n-1})\right)^+, \quad E_n^-=\left( K_3^- ( \Psi_+^{(0)n}-\Psi_+^{(0)n-1})\right)^-.
\end{equation*}
From \eqref{eq:iteration-psi}--\eqref{eq:iteration-phi}, the
difference of the values of the function after \(n+1\) and \(n\) 
times is
\begin{equation}
\label{eq:err1}
\Phi_-^{(L)n+1}-\Phi_-^{(L)n}= \frac{1}{(K_1)_-}\left(\frac{e^{-i
  \alpha L}E_{n+1}^-  }{(K_1)_+}\right)^-,
\end{equation}
and
\begin{equation}
\label{eq:err2}
\Psi_+^{(0)n+1}-\Psi_+^{(0)n}= \frac{1}{(K_3)_+}\left(\frac{e^{i
  \alpha L} D_n^+ }{(K_3)_-}\right)^+.
\end{equation}
Note that the forcing terms \(f_i^\pm\) do not influence the convergence.
In order to estimate the differences~\eqref{eq:err1} and~\eqref{eq:err2} in magnitude  
we need some inequalities for the Wiener--Hopf additive
decomposition. This has been addressed in \cite{My1} for all \(L_p\)
spaces, here we will need a very special case of that result which has a particularly simple form. We will use that if \( F(t)=F^{+}(t)+F^-(t),\) as in Theorem \ref{prob:jump-probl}, then
\begin{equation}
\label{eq:bound}
\lVert F^{\pm}\rVert_2\le\lVert F\rVert_2.
\end{equation}
 The same inequality can be observed from the fact that the map \(F
 \to F_\pm\) in \(L_2(\mathbb{R})\) is the Szeg\"o orthoprojector with the norm \(1\).
Since we will be looking at the \(L_2\) norm on different lines within the
strip of analyticity, the \(L_2\) norm of a function \(f(x+ia)\) on the line
\(\Im(\alpha)=a\) will be denoted by \[\lVert
f\rVert_2^a:=\left(\int\limits_{-\infty}^{\infty} |f(x+ia)|^2\,dx\right)^{1/2}.\]

As well as employing \eqref{eq:err1} to assert convergence we will need to
derive a relationship for \(E_{n+1}^-\) and \(E_{n}^-\). This is done
by a similar procedure used to derive \eqref{eq:err1} but considering
the other unknown in the scalar Wiener-Hopf equation. In other words
 the expression of the unknown plus and minus function are
 obtained. This is done explicitly below. 
Write \eqref{eq:eq32} and \eqref{eq:eq42} as they are used on \(n\)-th iteration:
\begin{align}
 K_1^-\Phi_-^{(L)n}=&  f_3^--\left(K_1^+\Phi_-^{(L)n}\right)^- +e^{-i
  \alpha L}\left(\left( K_3^-  \Psi_+^{ (0)n}\right)^-+f_4^-\right) , \label{eq:eq33} \\
K_3^+\Psi_+^{ (0)n}=& f_4^+-\left( K_3^-  \Psi_+^{ (0)n}\right)^++e^{i \alpha L}\left( \left( K_1^+\Phi_-^{(L)n-1}\right)^++f_3^+\right). \label{eq:eq43}
 \end{align}
We can add \(K_1^+\Phi_-^{(L)n}\) to both sides of \eqref{eq:eq33} and
\( K_3^-  \Psi_+^{ (0)n}\) to both sides of \eqref{eq:eq43} to obtain
\begin{align}
 K_1\Phi_-^{(L)n}=&  f_3^-+\left(K_1^+\Phi_-^{(L)n}\right)^+ +e^{-i
  \alpha L}\left(\left( K_3^-  \Psi_+^{ (0)n}\right)^-+f_4^-\right) , \label{eq:eq34} \\
K_3\Psi_+^{ (0)n}=& f_4^++\left( K_3^-  \Psi_+^{ (0)n}\right)^-+e^{i \alpha L}\left( \left( K_1^+\Phi_-^{(L)n-1}\right)^++f_3^+\right). \label{eq:eq44}
 \end{align}
And if we consider the difference between the consecutive iterations
we derive
\begin{align}
 K_1(\Phi_-^{(L)n}-\Phi_-^{(L)n-1})=& D_n^+  +e^{-i
  \alpha L}E_n^-, \label{eq:eq34} \\
K_3( \Psi_+^{(0)n+1}-\Psi_+^{(0)n})=& E_{n+1}^-+e^{i \alpha L}D_n^+ . \label{eq:eq44}
 \end{align}
This makes the coupling between the equations explicit. Note that
\begin{equation*}
D_n^+=(K_1)_+\left(\frac{e^{-i
  \alpha L}E_n^-}{(K_1)_+}\right)^+, \quad E_{n+1}^-=(K_3)_-\left(\frac{e^{i
  \alpha L}D_n^+}{(K_3)_-}\right)^-,
\end{equation*}
where \(D_n^+\) can be derived by solving the scalar Wiener--Hopf equation
\eqref{eq:eq34} with the unknown minus function
\(\Phi_-^{(L)n}-\Phi_-^{(L)n-1}\), unknown plus function  \(D_n^+\)
and \(e^{-i
  \alpha L}E_n^-\) is assumed to be known from the previous
iterations. Similarly \(E_{n+1}^-\) is obtained by solving the  scalar
Wiener--Hopf equation \eqref{eq:eq44}.

We will obtain estimates of the size of \(\lVert D_n^+\rVert_2^a\).
 Let \(\max_{x \in
  \mathbb{R}}|(K_1)_+(x+ai)|=d_1\) and \(\max_{x \in
  \mathbb{R}}|(K_1)_+^{-1}(x+ai)e^{-i L(x+ai)}|=\epsilon_1\). Note that \(K_1\) is bounded since
it is in \(L_2(\mathbb{R}) \cap \text{ H\"{o}lder}\) and note
that \(K_1\) is non-zero. Then using \eqref{eq:bound} we have
\begin{equation}
\label{eq:line1}
\lVert D_n^+\rVert_2^a\le d_1\left\lVert\left(\frac{e^{-i
  \alpha L}E_n^-}{(K_1)_+}\right)^+\right\rVert_2^a\le d_1\left\lVert\frac{e^{-i
  \alpha L}E_n^-}{(K_1)_+}\right\rVert_2^a \le d_1 \epsilon_1 \lVert E_n^-\rVert_2^a.
\end{equation}
Similarly, defining \(\max_{x \in \mathbb{R}}(K_3)_-(x+bi)=d_2\)
and \(\max_{x \in \mathbb{R}}e^{i
  L(x+bi)}(K_3)_-^{-1}(x+bi)=\epsilon_2\), we obtain
\begin{equation}
\label{eq:line2}
\lVert E_{n+1}^- \rVert_2^b\le d_2\left\lVert \left(\frac{e^{i
  \alpha L}D_n^+}{(K_3)_-}\right)^- \right\rVert_2^b\le d_2\left\lVert\left(\frac{e^{i
  \alpha L}D_n^+}{(K_3)_-}\right)\right\rVert_2^a \le d_2 \epsilon_2
\lVert D_n^+ \rVert_2^a.
\end{equation}
Next, we  note that the following is true
\begin{equation}
\label{eq:line}
\lVert E_{n+1}^- \rVert_2^a \le \lVert E_{n+1}^- \rVert_2^b, \qquad
\lVert D_n^+\rVert_2^b \le \lVert D_n^+\rVert_2^a.
\end{equation}
 This is intuitively clear since \(E_{n+1}^-\) is further from
 singularities on the line \(\Im(\alpha)=a\) than on \(\Im(\alpha)=b\)
 and the other way round for \(D_n^+\). The inequalities follows from the Poisson formula  for the real line~\cite{Nikolski}*{Cor~6.4.1}
 and H\"{o}lder inequality.
Combining \eqref{eq:line} with inequalities \eqref{eq:line1} and \eqref{eq:line2} we obtain the key result
for demonstrating convergence:
\begin{equation}
\label{eq:liner}
\lVert E_{n+1}^- \rVert_2^a \le d_1  d_2 \epsilon_2\epsilon_1 \lVert E_{n}^- \rVert_2^a, \qquad
\lVert D_{n+1}^+\rVert_2^b \le d_1  d_2 \epsilon_2\epsilon_1 \lVert D_n^+\rVert_2^a.
\end{equation}
The convergence of the procedure is shown in the next theorem.
\begin{thm} For sufficiently large \(L\) there exists a constant \(q< 1\) such that
\begin{equation*}
\lVert\Phi_-^{(L)n+1}-\Phi_-^{(L)n}\rVert_2^a=q\lVert\Phi_-^{(L)n}-\Phi_-^{(L)n-1}\rVert_2^a,
\end{equation*}
for all \(n\) and, hence, the error at the \(n\)-th iteration is 
\begin{equation*}
\lVert\Phi_-^{(L)n}-\Phi_-^{(L)}\rVert_2^a =\frac{q^n}{1-q}\lVert \Phi_-^{(L)0}-\Phi_-^{(L)1}\rVert_2^a.
\end{equation*}
Analogous statements are true for \( \Psi_+^{(0)n}\).
\end{thm}

\begin{proof}
First consider \eqref{eq:err1} on the line \(x+ai\). Let \(\max_{x \in
  \mathbb{R}}|(K_1)_-^{-1}(x+ai)|=c_1\), note that \(K_1\) is non-zero on the strip so
the constant is well defined. Then, using \eqref{eq:bound}, we have
\begin{equation}
\lVert\Phi_-^{(L)n+1}-\Phi_-^{(L)n}\rVert_2^a \le c_1 \left\lVert \left(\frac{e^{-i
  \alpha L}E_{n+1}^-}{(K_1)_+}\right)^-\right\rVert_2^a \le c_1\epsilon_1\lVert E_{n+1}^- \rVert_2^a. 
\end{equation}
This implies that 
\begin{equation}
\lVert\Phi_-^{(L)n+1}-\Phi_-^{(L)n}\rVert_2^a \le \frac{\lVert E_{n+1}^- \rVert_2^a }{\lVert E_{n}^- \rVert_2^a}\lVert\Phi_-^{(L)n}-\Phi_-^{(L)n-1}\rVert_2^a.
\end{equation}
Hence we obtained the desired result with \(q=d_1  d_2
\epsilon_2\epsilon_1\) using \eqref{eq:liner}. The contraction mapping theorem can be applied if
\(q<1\) to show that the iterations converge to the exact solution
\(\Phi_-^{(L)}\).  Note that we can make \(\epsilon_1\) and
\(\epsilon_2\) (and hence \(q\)) arbitrarily small by taking \(L\) sufficiently large.

Similarly, we consider \eqref{eq:err2} on the line \(x+bi\). 
In the same manner we derive
\begin{equation}
\lVert \Psi_+^{(0)n+1}-\Psi_+^{(0)n}\rVert_2^b\le q \lVert \Psi_+^{(0)n}-\Psi_+^{(0)n-1}\rVert_2^b.
\end{equation}
 \end{proof}

Note that the procedure converges for all possible initial functions. But the accuracy of the initial step determines how fast
the small desired error will be achieved. 

\section{Examples}
\label{sec:ex}

In this section the method proposed in this paper will be illustrated
numerically. Two examples will be considered, in both cases the exact
solutions are known and will be compared with the outcomes of the
iterative procedure. The first example is of the type \eqref{eq:main}
and is chosen to be
 as simple as possible. 
The second example has been studied by other researchers in connection to
integral equations. The resulting Wiener--Hopf system is more general than
\eqref{eq:main}, but it can also be reduced to solving equations
similar to
\eqref{eq:eq32} and \eqref{eq:eq42} and hence the derivations in this
paper apply. A more involved examples of application of the proposed
method is found in \cite{my+Lorna} and is used in the setting of
acoustics, see Section \ref{sec:Motivation} for more details. Note
that in the latter case no exact solution is known.
\begin{ex}
\begin{figure}[htbp]
\centering
\includegraphics[scale=0.63,angle=0]{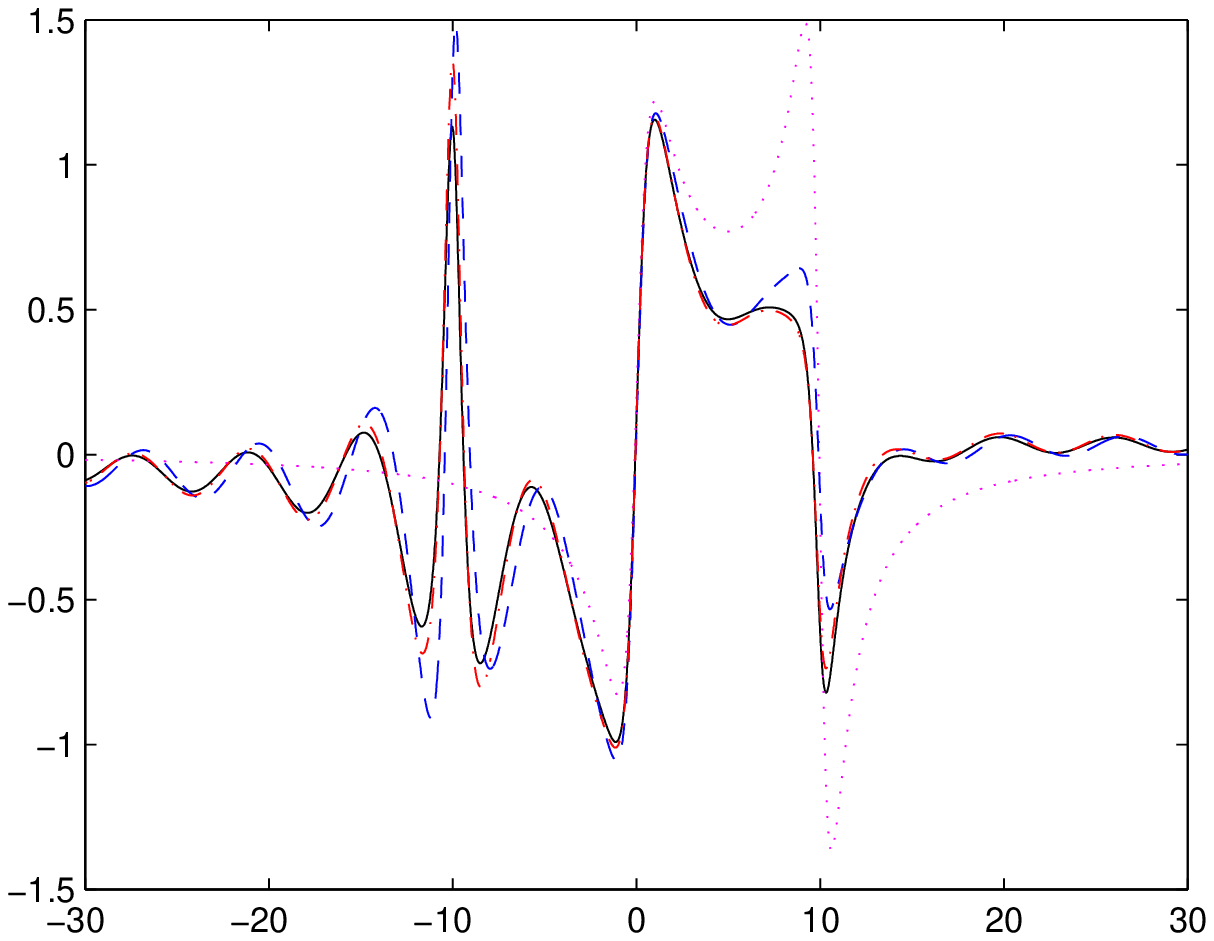}
\includegraphics[scale=0.63,angle=0]{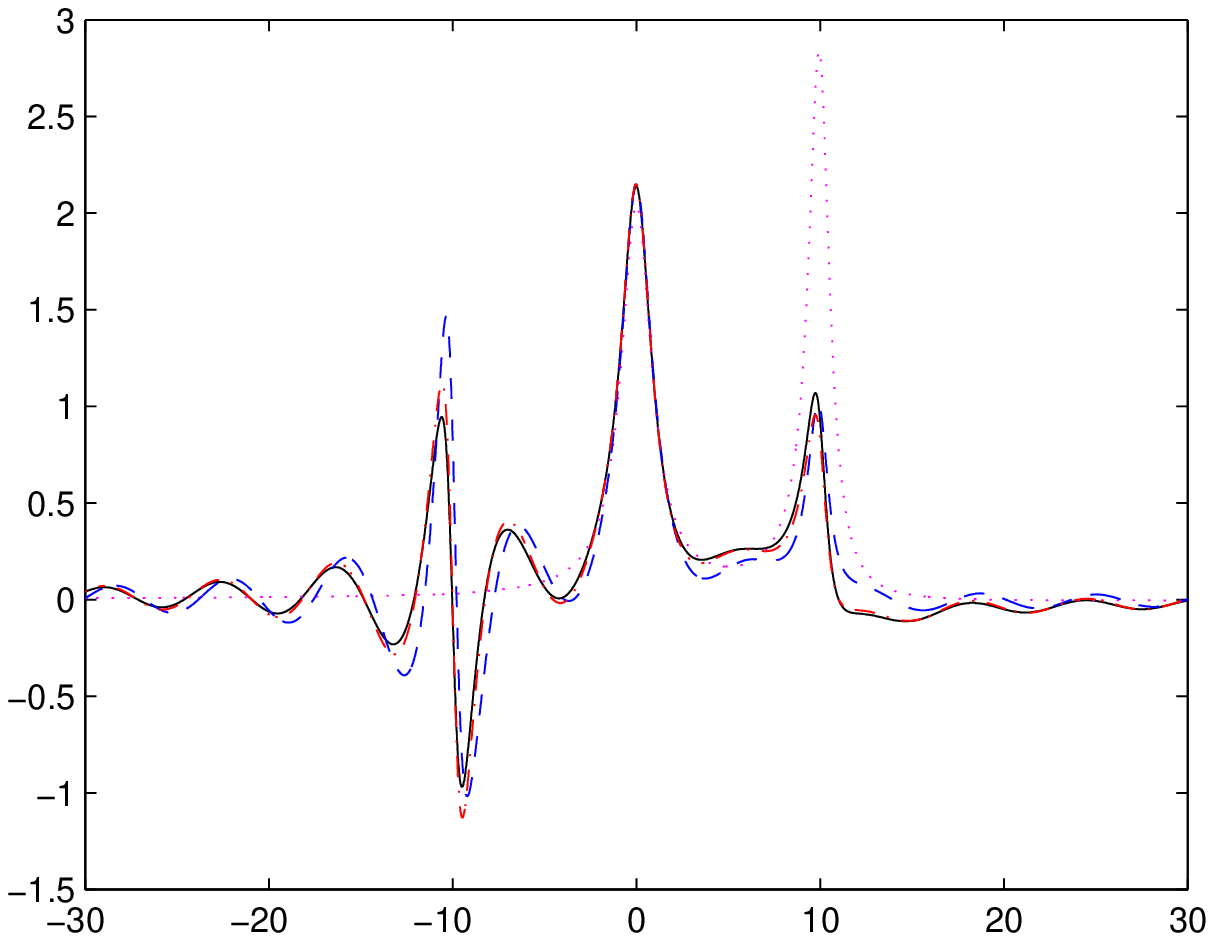}
\includegraphics[scale=0.63,angle=0]{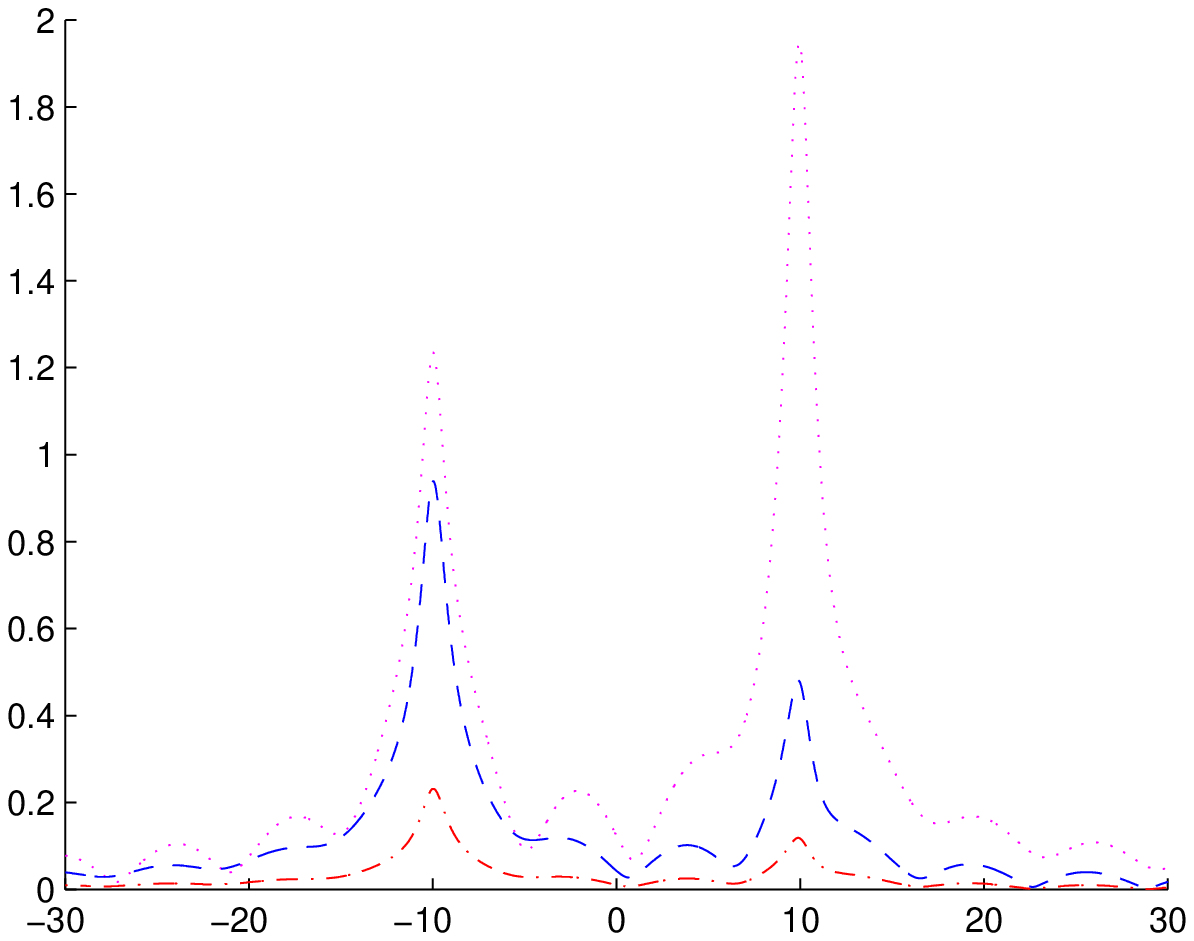}
\caption{Showing the real part (top figure) and imaginary part
  (middle figure) of  $\Phi_-^{(L)}$ (solid black line),
  $\Phi_-^{(L)0}$ (dotted line), $\Phi_-^{(L)1}$ (blue dashed line)
  and $\Phi_-^{(L)2}$ (red dashed line). The bottom figure shows the
  decrease in the
  absolute value of $\Phi_-^{(L)}-\Phi_-^{(L)n}$. The parameters are $\lambda
=0.7+10i $ and $L=1$.}
\label{fig:3}
 \end{figure}
\begin{figure}[htbp]
\centering
\includegraphics[scale=0.47,angle=0]{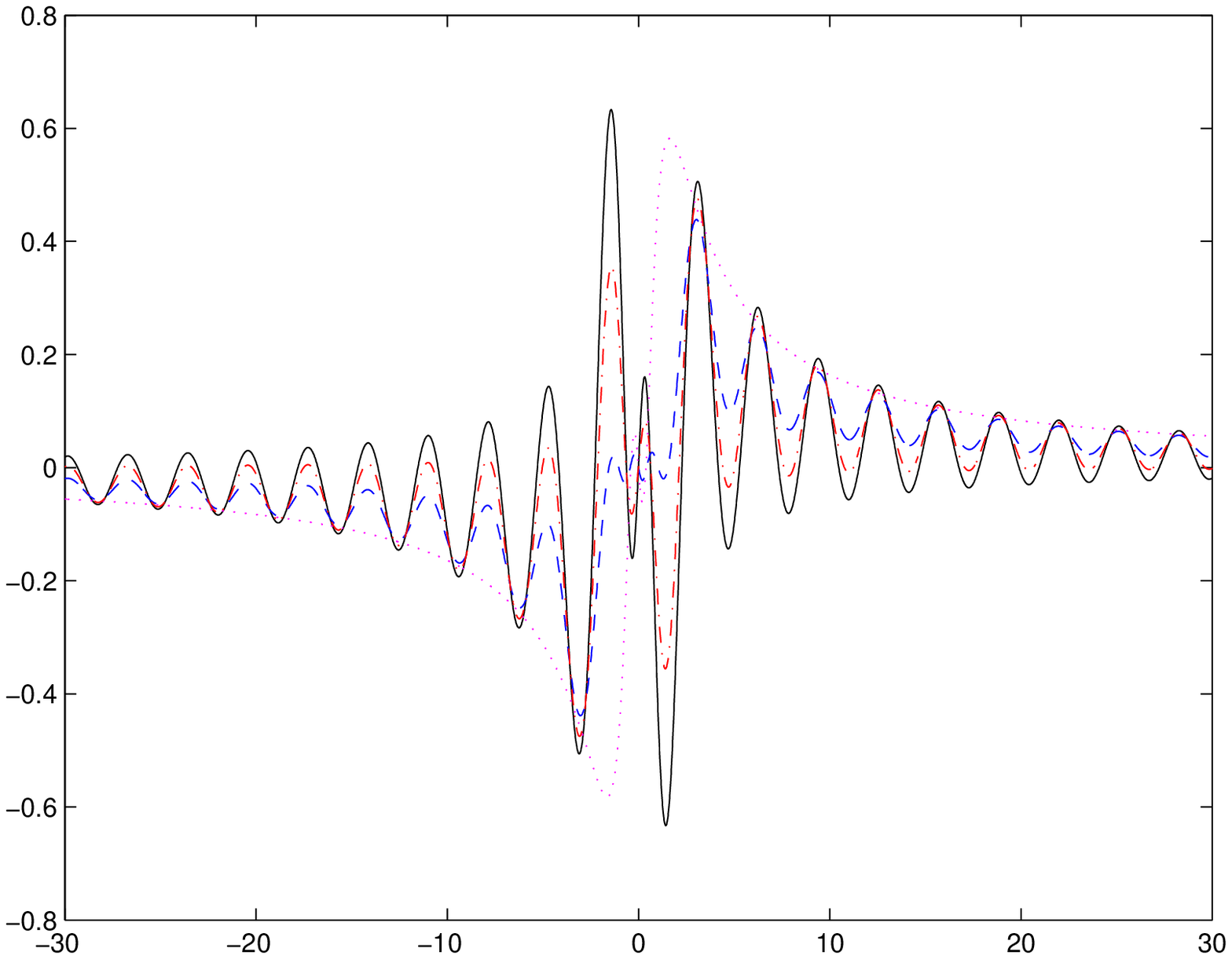}
\includegraphics[scale=0.47,angle=0]{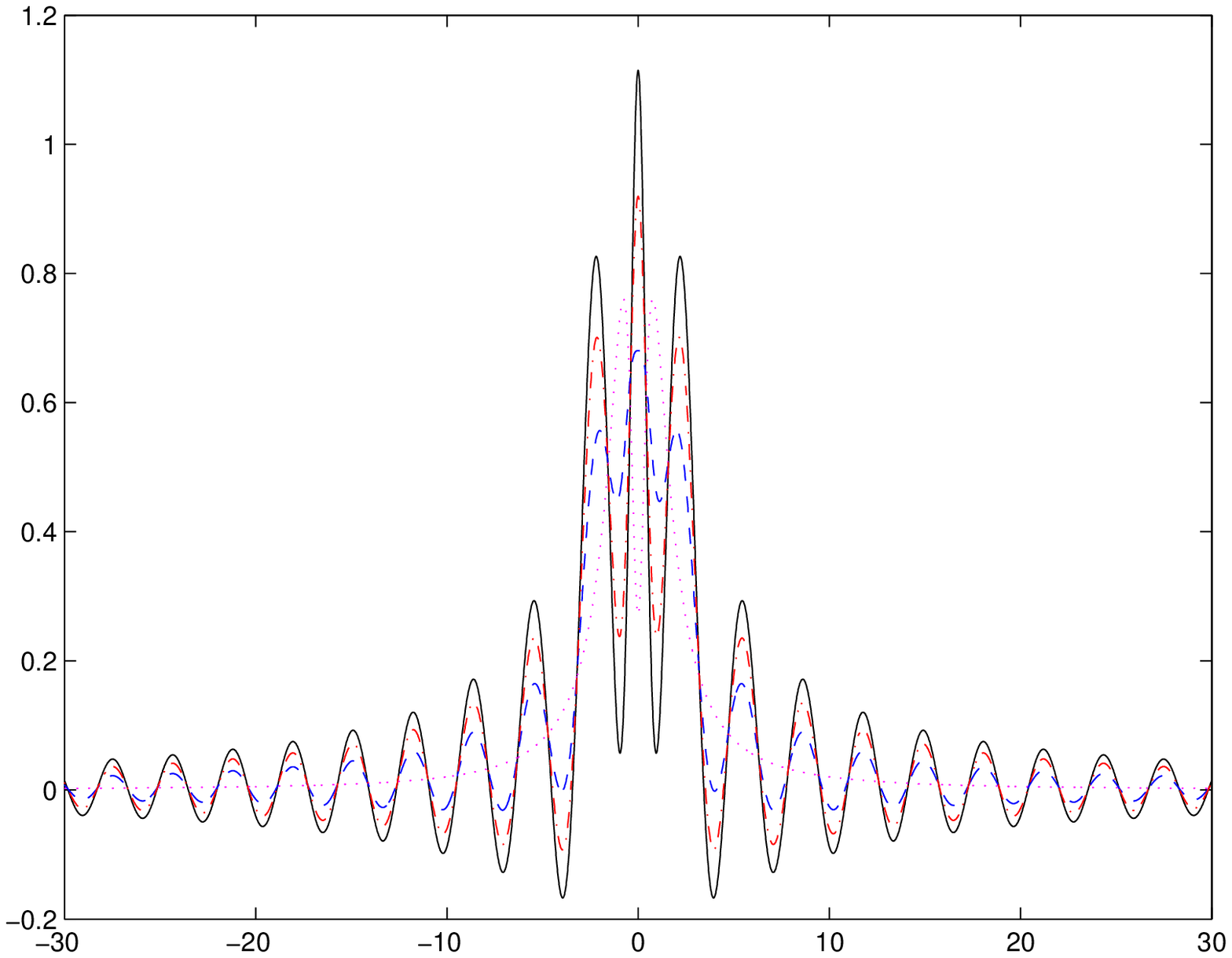}
\includegraphics[scale=0.47,angle=0]{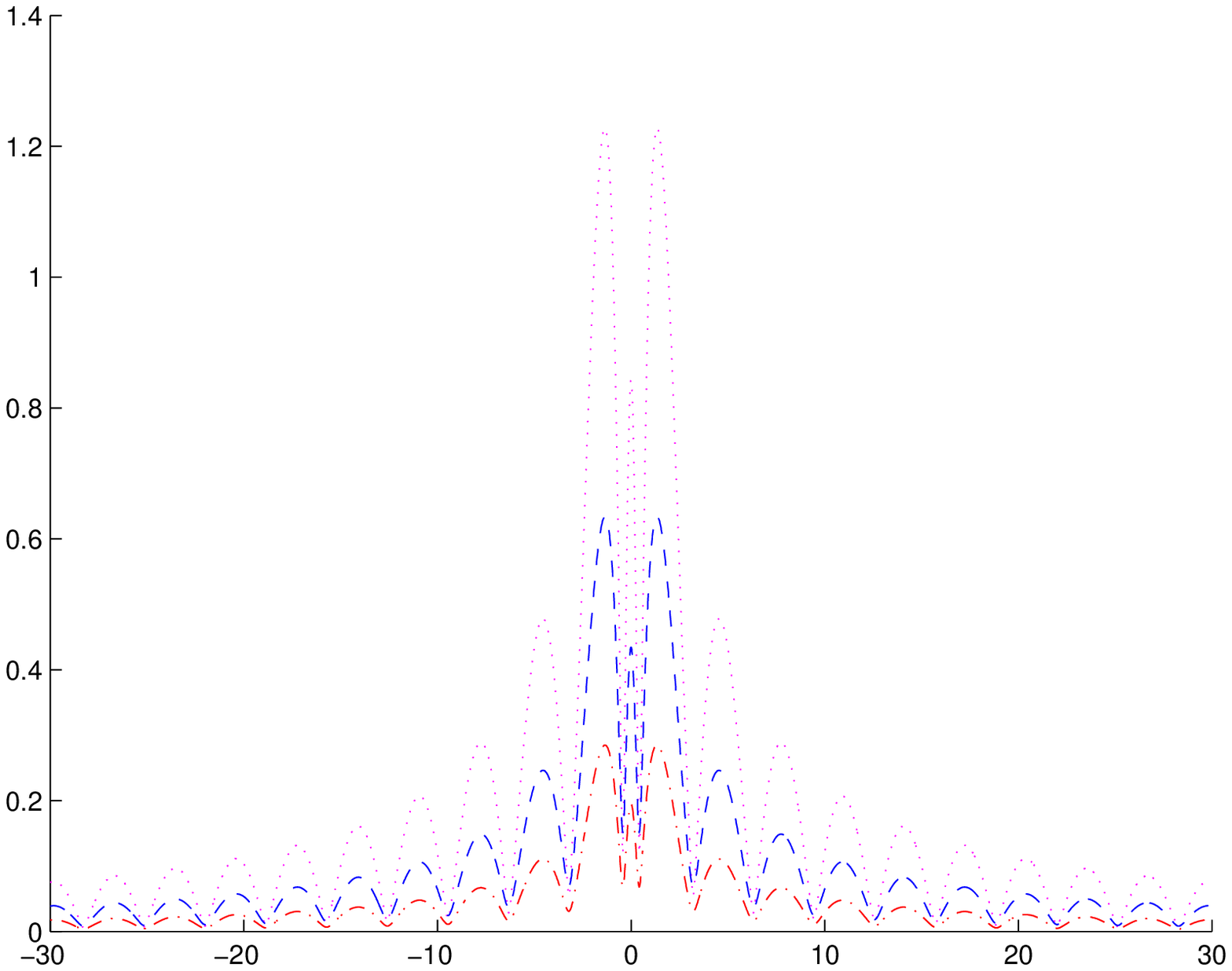}
\caption{Showing the real part (top figure) and imaginary part
  (middle figure) of  $\Phi_-^{(L)}$ (solid black line),
  $\Phi_-^{(L)0}$ (dotted line), $\Phi_-^{(L)1}$ (blue dashed line)
  and $\Phi_-^{(L)2}$ (red dashed line). The bottom figure shows the
  decrease in the
  absolute value of $\Phi_-^{(L)}-\Phi_-^{(L)n}$. The parameters are $\lambda
=0.2 $ and $L=2$.}
\label{fig:2}
 \end{figure}

The first numerical example will be the simplest possible in order to
illustrate the theory. Consider \eqref{eq:main} with
\begin{equation*}
\begin{pmatrix}
 \frac{0.5}{(\alpha-i\lambda)(\alpha+i\lambda)}+1 &  B_+(\alpha)e^{i \alpha L} \\
 e^{-i \alpha L} & 0
 \end{pmatrix},
\end{equation*}
where \(\lambda\) is a complex parameter and  \(B_+(\alpha)\) is an
arbitrary function satisfying the conditions stated in Section~\ref{sec:1}. For the forcing terms take
\begin{equation*}
f_4^-(\alpha)=f_3^-(\alpha)=\frac{1}{\alpha-i}, \quad \text{and} \quad f_4^+(\alpha)=f_3^+(\alpha)=\frac{1}{\alpha+2i}.
\end{equation*}
In this example \(a=b=\Re (\lambda)\) and so we
require \(\Re (\lambda)>0\).  We have  that
\begin{equation*}
K_3(\alpha)=K_1(\alpha)=\frac{0.5}{(\alpha-i\lambda)(\alpha+i\lambda)}+1.
\end{equation*}
Then equation \eqref{eq:eq32} and \eqref{eq:eq42} become
\begin{align}
 K_1(\alpha)\Phi_-^{(L)}(\alpha)=&\frac{1}{\alpha-i}
 +\frac{k_2}{\alpha+i\lambda}+e^{-i \alpha L}\left(
\frac{k_1}{\alpha-i\lambda}+\frac{1}{\alpha-i}\right)\label{eq:ex1},\\
K_1(\alpha)\Psi_+^{ (0)}(\alpha)=&\frac{1}{\alpha+2i}-\frac{k_1}{\alpha-i\lambda}+e^{i\alpha L}\left(
\frac{1}{\alpha+2i}-\frac{k_2}{\alpha+i\lambda}\right)\label{eq:ex12},\
\end{align}
and we are required to find \(\Phi_-^{(L)}(\alpha)\) and \(\Psi_+^{ (0)}(\alpha)\).
In this simple case we can exactly solve these coupled equations and
find  constants \(k_1\) and \(k_2\) 
explicitly.  Define
\begin{equation*}
b=e^{-\lambda L}, \quad c=2i\lambda f_1^-(-i\lambda),\quad \text{and} \quad
d=2i\lambda f_1^+(i\lambda),
\end{equation*}
then the constants are given by
\begin{equation*}
k_1=\frac{d-bc}{1-b},\quad \text{and} \quad  k_2=\frac{c-bd}{1-b}. 
\end{equation*}
Next, we will solve these coupled Wiener--Hopf equations using an iterative
procedure described in Section \ref{sec:Iterative}. The first step is
to neglect the the term with \(e^{-i \alpha L}\) in \eqref{eq:ex1}
and hence uncouple the equations and obtain a value for \(k_2^{(0)}\). This
value is then substituted into \eqref{eq:ex12} to obtain \(k_1^{(0)}\).
The iterative procedure leads to the following  values
\begin{equation*}
 k_1^{(n)}=d+bd-bk_2^{(n)}, \quad  k_2^{(n)}=c+bc-bk_1^{(n-1)},\quad \text{with} \quad k_2^{(0)}=c .
\end{equation*}
This converges to the actual solution as long as \(b<1\).
In most cases the convergence is very fast and the line of the first
iteration is indistinguishable to the approximate solution. In the
cases when the convergence is slow (small \(L\)
and small \(\Re(\lambda)\))
the first iteration still retains some features of the solution, see
Figure \ref{fig:3} and \ref{fig:2}. The first guess has the correct
overall shape and all the consecutive iterations have the maxima and
the minima in the right places even when the magnitude is quite
different. Also note that even in the cases of slow convergence, 
the second iteration is already  close to the actual solution.

In this special case it is possible to say more about the convergence
of the solution compared to the general case (Section
\ref{sec:Convergence}). It is easy to see that
\begin{equation*}
\Phi_-^{(L)n+1}-\Phi_-^{(L)n}= \frac{1}{K_1}\left(\frac{k_2^{(n+1)}-k_2^{(n)}}{\alpha+i\lambda}+e^{-i \alpha L}
\frac{k_1^{(n+1)}-k_1^{(n)}}{\alpha-i\lambda}\right).
\end{equation*}
Hence
\begin{equation*}
\Phi_-^{(L)n+1}-\Phi_-^{(L)n}= \frac{(k_2^{(n+1)}-k_2^{(n)})}{K_1}\left(\frac{1}{\alpha+i\lambda}+e^{-i \alpha L}
\frac{-b}{\alpha-i\lambda}\right).
\end{equation*}
This means that 
\begin{equation*}
\lVert \Phi_-^{(L)n+1}-\Phi_-^{(L)n}\rVert_2 \le b^2 \lVert \Phi_-^{(L)n}-\Phi_-^{(L)n-1}\rVert_2 .
\end{equation*}
This can be verified numerically and is illustrated in Figure \ref{fig:1}.
\begin{figure}[htbp]
\centering
  \includegraphics[scale=0.9,angle=0]{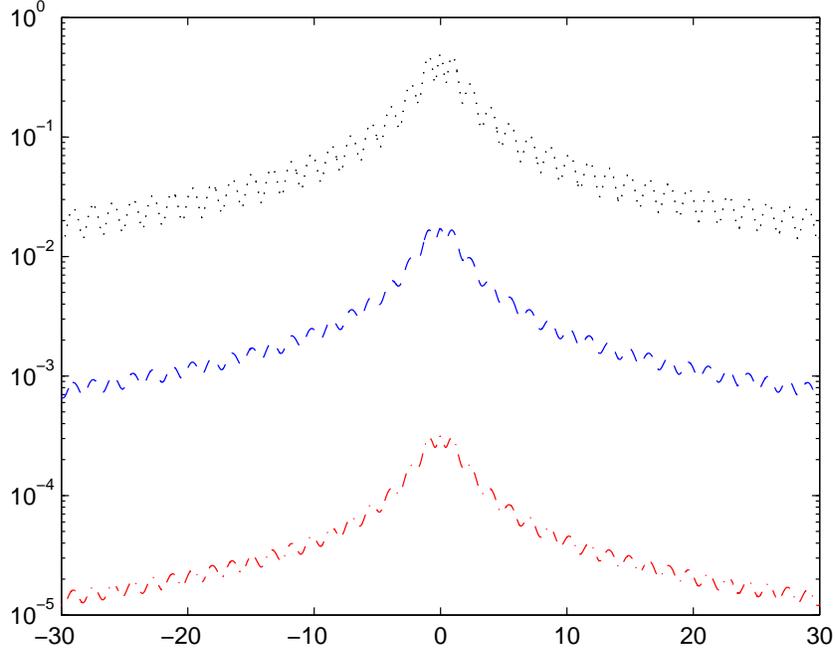}
\caption{Showing the decrease in the absolute value of $\Phi_-^{(L)}-\Phi_-^{(L)n}$. }
\label{fig:1}
 \end{figure}


%

\end{ex}

\begin{ex}
The next example is more complicated and arises from an  integral
equation. The ability to solve integral equations is important in many applications \cite{Vekua, Musk}.
Consider the following (one-sided) integral equation
\begin{equation*}
u(x)=\lambda \int^{\infty}_0 k(x-t)u(t) dt +f(x), \quad 0<x<\infty,
\end{equation*}
with the matrix kernel given by
\begin{equation*}
k(x)=
\begin{pmatrix}
e^{-|x|}  & e^{-|x-L|} \\
e^{-|x+L|}  & e^{-|x|}
 \end{pmatrix},
\end{equation*}
with \(f(x)\) a forcing function, \(\lambda\), \(L\)
real parameters and \(u(x)=(u^{(1)}(x), u^{(2)}(x))^{T}\) to be determined.  This system was considered in \cite{Abr_exp} and
more recently in \cite{Antipov}. This integral equation can be reduced to
a Wiener--Hopf equation by extending the range of \(x\)
to the whole real line and applying the Fourier transform. The
resulting Wiener--Hopf equation is of a more general type than \eqref{eq:main}.  It has been shown
\cite{Antipov}  that the solution could be reduced to
finding two constants \(C_1\)
and \(C_2\)
by the Wiener--Hopf method. In this example it is possible to obtain the
exact solution, which  provides a good way to test out the ideas that are
introduced in this paper. Once the solution to the Wiener--Hopf
equation is obtained, the inverse Fourier transform will provide the
solution to the integral equation.
\begin{figure}[htbp]
\centering
\includegraphics[scale=0.7,angle=0]{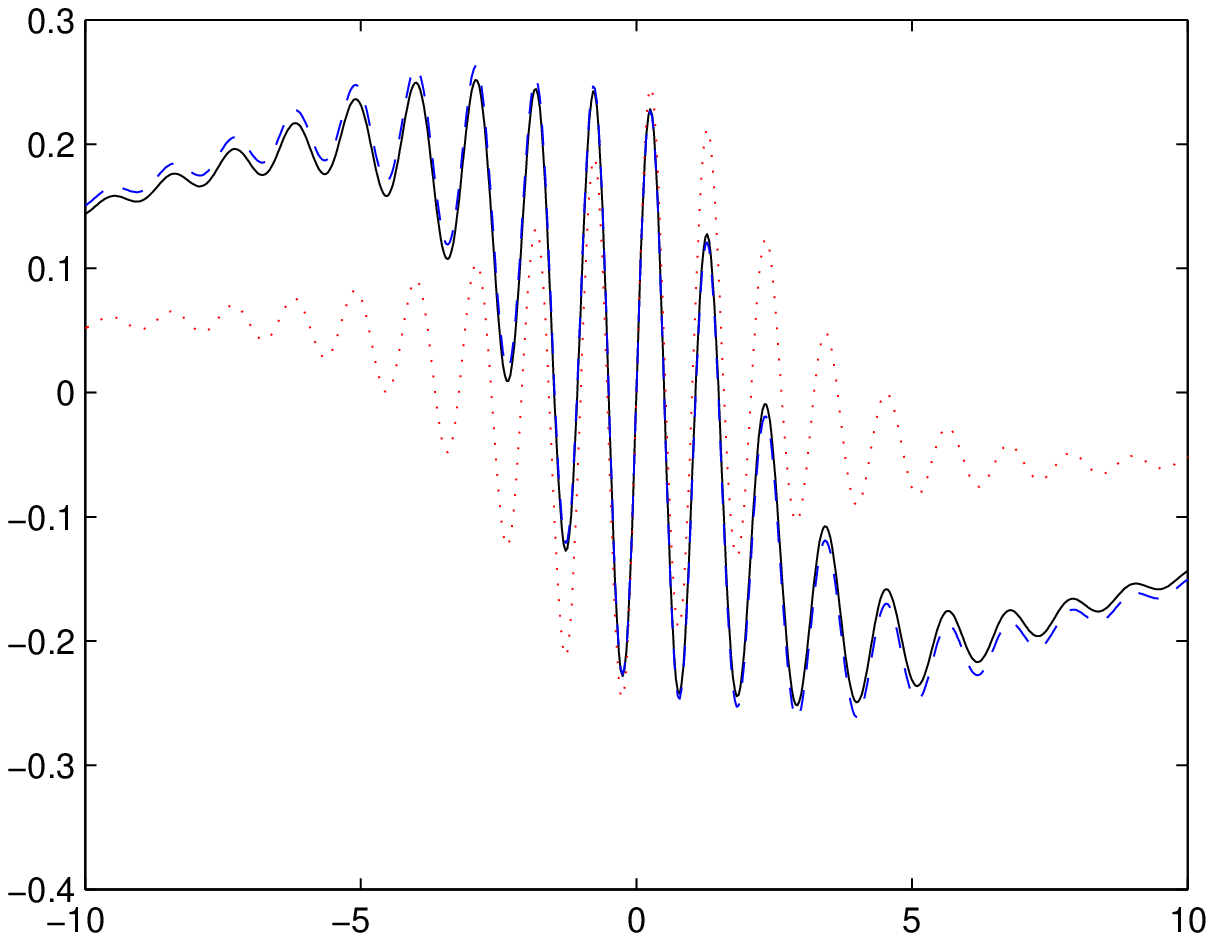}
\includegraphics[scale=0.7,angle=0]{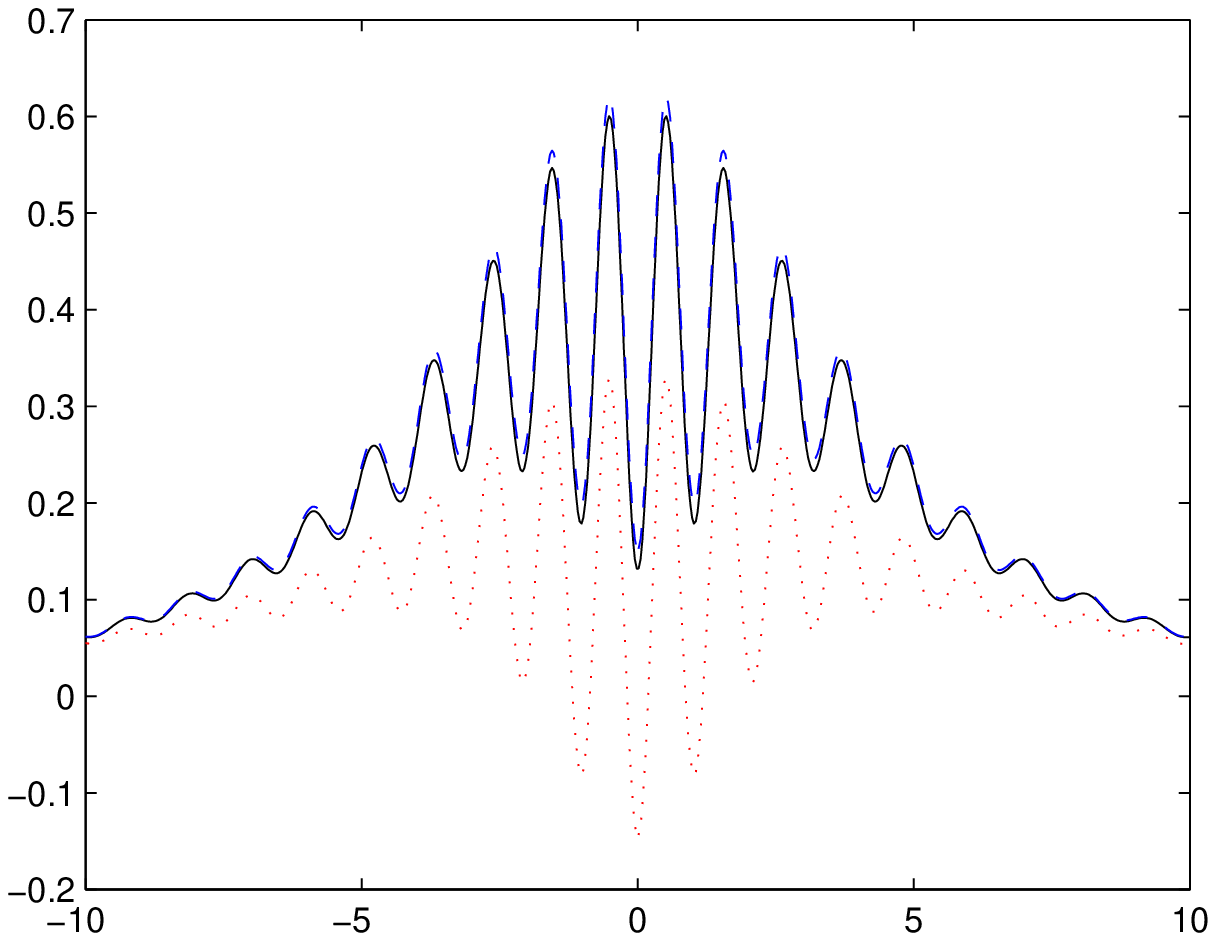}
\caption{Showing the real part (top figure) and imaginary part
  (bottom figure) of  $U_+^{(1)}$ (solid black line), $U_+^{(1)0}$ (red
dotted line)  and $U_+^{(1)1}$ (blue dashed line). The parameters are $\lambda
= -15$ and $L=0.04$.}
\label{fig:4}
 \end{figure}
First we will need to define some
functions. Let us take two constants \(\lambda_0=\sqrt{1-2\lambda}\)
and \(\lambda_1=\sqrt{1-4\lambda}\) for some  \(\lambda \in
(-\infty,0.25]\), then we define
\begin{align*}
M_-(\alpha)=& \frac{\alpha-i\lambda_1}{\alpha-i\lambda_0}, \quad M_+(\alpha)=\frac{\alpha+i\lambda_0}{\alpha+i\lambda_1}, \\
K_-(\alpha)= & \frac{\alpha-i\lambda_0}{\alpha-i}, \quad K_+(\alpha)=  \frac{\alpha+i}{\alpha+i \lambda_0}.
\end{align*}
The forcing is taken as \(F_1^+(\alpha)=\frac{1}{\alpha-2i}\) and
\(F_2^+=0\) and we use the following additive splittings 
\begin{align*}
L_1^+(\alpha)-L_1^-(\alpha)&=\frac{\alpha-i}{(\alpha-i \lambda_0)(\alpha-2i)},\\
L_2^+(\alpha)-L_2^-(\alpha)&=\frac{2 \lambda e^{-i \alpha L} }{(\alpha-2i)(\alpha+i\lambda_0)(\alpha-i\lambda_1)}.
\end{align*}

The Wiener--Hopf method reduces the solution to equations similar to
\eqref{eq:eq32} and \eqref{eq:eq42}. The unknown functions
\(U_-^{(2)}(\alpha)\) and \(U_+^{(1)}(\alpha)\) are the half-range Fourier
transforms \eqref{eqn:halfF} of \(u^{(2)}(\alpha)\) and
\(u^{(1)}(\alpha)\). They are given by

\begin{align*}
U_-^{(2)}(\alpha)=& M_-(\alpha) \left(L_2^-(\alpha)
 +\frac{C_2}{\alpha+i\lambda_0}+\frac{2\lambda K_-(\alpha) e^{-i \alpha L}
 }{(\alpha+i\lambda_1)(\alpha-i\lambda_0)}\left(L_1^-(\alpha)+\frac{C_1}{\alpha-i\lambda_0} \right)\right),\\
U_+^{(1)}(\alpha)=&K_+ (\alpha)\left(L_1^+(\alpha)
 +\frac{C_1}{\alpha-i\lambda_0}+\frac{2\lambda M_+(\alpha)e^{i \alpha L}
 }{(\alpha+i)(\alpha-i\lambda_0)}\left(L_2^+(\alpha)+\frac{C_2}{\alpha+i\lambda_0} \right)\right).
\end{align*}
The constants can be found explicitly 
\begin{align*}
C_1=\frac{d_1+d_2b}{b^2+1}, \qquad
C_2=\frac{d_2-d_1b}{b^2+1},
\end{align*}
where
\begin{align*}
b=\frac{2\lambda e^{i L\lambda_0}
 }{(\lambda_0+1)(\lambda_0+\lambda_1)},\qquad
d_1=2ib \lambda_0 L_2^+(i\lambda_0), \qquad d_2=2ib \lambda_0 L_1^-(-i\lambda_0).
\end{align*}
 By employing the iterative procedure we obtain 
\begin{align*}
C_1^{(0)}=&d_1, & C_2^{(1)}=&d_2-d_1b,\\
C_1^{(1)}=&d_1+d_2b-d_1b^2, &
C_2^{(1)}=&d_2-d_1b-d_2b^2+d_1b^3. \\
C_1^{(2)}=&d_1+d_2b-d_1b^2-d_2b^3+d_1b^4.
\end{align*}
In fact, each iteration adds two more terms in the Taylor expansion of
the constants. For example, for \(\lambda = 0.1\)
and \(L=0.0001\)
the maximum error for \(U_-^{(2)0}\)
is \(10^{-4}\)
and for \(U_-^{(2)1}\)
is \(10^{-6}\).
Note that the convergence speed of iterations depends on \(b^2\),
in the same way as in the previous example. It fact, it is small for
all values of \(\lambda\)
and \(L\) and \(b^2\le 0.18\).
This shows that in this example the convergence is extremely fast for
all values of the parameters.  Another factor, which influences the
error of the iterative solution, is suitability of the initial value. For
example, with \(\lambda = -15\)
and \(L=0.04\)
the initial guess $U_+^{(1)0}$ is very bad, see Figure \ref{fig:4}
the red dotted line. But since \(b^2= 0.0747\),
even the first iteration $U_+^{(1)0}$ is already very accurate, see Figure
\ref{fig:4} the blue dashed line.

\end{ex}

\section{Conclusion}

The Wiener--Hopf method is a powerful tool for solving boundary value
problems, which has been applied in an impressive array of
situations. In the case of scalar Wiener--Hopf equations the solution is
algorithmic. In the matrix case, the  matrix
Wiener--Hopf factorisation cannot be in general obtained and hinders
the use of the method. The present paper presents an algorithmic
way of bypassing this step, for a class of matrix functions~\eqref{eq:main}. Only
scalar Wiener--Hopf splittings are used in an iterative procedure. 
 We provide the 
conditions for the convergence of iterations to the exact
solution. This also enables us to estimate  the error at each
iteration. The numerical examples show that in most cases only a few
iterations are required.

It is clear that this method could be applied to a wider class of
Wiener--Hopf systems than~\eqref{eq:main}. For instance, the matrix in
Example 2 is not triangular, however it still can be solved using our
methods. Further work could be done to extend this method to
non-triangular matrix functions with exponential factors.

This method has been motivated by applications and was already used in
\cite{my+Lorna}. There is  scope for using this method in a variety of
boundary value problems with finite geometries or more than one
transition point in the boundary conditions. Possible applications
could be in different areas such as electromagnetism, fracture
mechanics and economics. It has also been shown in numerical Example 2
that some integral equations can also be solved, this even opens a
larger scope of applications.

\section{Acknowledgements}

I am grateful to Professor Peake for suggesting and discussing the
acoustic problem which motivated this research. I also benefited from
useful discussions with Professors Mishuris and Rogosin. I acknowledge
support from Sultan Qaboos Research Fellowship at Corpus Christi
College at University of Cambridge.

\bibliography{newgeometry}

\end{document}